\documentclass[11pt]{article}

\usepackage{amssymb,amsmath,amsfonts}
\usepackage{float}
\usepackage{graphicx}
\usepackage{etoolbox}
\usepackage{xcolor}
\usepackage{listings}
\usepackage{setspace}
\usepackage[most]{tcolorbox}
\usepackage[margin=1in]{geometry}
\usepackage{tikz,lipsum,lmodern}
\usepackage{enumitem}
\usepackage{amsthm}

\newtheorem{theorem}{Theorem}[section]
\newtheorem{lemma}{Lemma}[section]
\newtheorem{definition}{Definition}[section]

\title{Perturbation Analysis and Neural Network-Based Initial Condition Estimation for the Sine-Gordon Equation}

\author{
Junhong Ha\\
School of Liberal Arts,\\
Korea University of Technology and Education,\\
Cheonan 31253, Korea\\
\texttt{hjh@koreatech.ac.kr}
\and
Sudeok Shon\\
School of Architecture Engineering,\\
Korea University of Technology and Education,\\
Cheonan 31253, Korea\\
\texttt{sdshon@koreatech.ac.kr}
}

\date{}

\vspace{0.5cm}

\begin{document}
\maketitle

\begin{center}
\begin{minipage}[c]{12cm}
{\small ABSTRACT.
The phase difference in Josephson junctions and superconducting systems follows the sine-Gordon equation, which describes soliton dynamics and phase evolution. These phase differences have also been explored as potential mechanisms for realizing quantum phase gates. Many integrable soliton solutions to the sine-Gordon equation are known. In this study, we consider solutions to the perturbed sine-Gordon equation as perturbations around soliton profiles and develop a mathematical framework to analyze the stability, with a particular focus on the boundedness of these perturbations. This analysis relies on initial conditions being sufficiently close to a soliton shape. However, identifying such initial states with precision is often a challenge. To address this, we introduce a data-driven method using neural networks to estimate the initial conditions of the system.
}
\end{minipage}
\end{center}

\vspace{0.5cm}

{\small 2020 Mathematics Subject Classification 35Q51, 35R30, 68T07}

{\small Key words and phrases  sine-Gordon equation, soliton dynamics, perturbation, identification, neural networks}

\section{Introduction}

The sine-Gordon equation is a fundamental nonlinear partial differential equation that models a wide range of physical systems, including wave propagation in Josephson junctions, nonlinear optics, and soliton dynamics in condensed matter physics \cite{Scott2000}, \cite{RyogoHirota1972}, \cite{Scott1976}. In superconducting circuits, particularly those involving Josephson junctions, the phase difference between superconducting wave functions plays a critical role in determining system behavior. This phase difference, denoted \( u(x,t) \), evolves according to the dimensionless sine-Gordon equation
\[
\frac{\partial^2 u}{\partial t^2} - \frac{\partial^2 u}{\partial x^2} + \sin u = g(x,t),
\]
where \( g(x,t) = I_b + H \), with \( I_b \) representing an external current and \( H \) denoting the gradient of a magnetic field. Perturbative effects such as damping and external forcing further influence the system dynamics \cite{Ustinov2001}, \cite{Gulevich2006}, \cite{Levi1978}, \cite{Gonzalez2002}, \cite{Nakajima1974}, \cite{SALERNO1983}. In fact, the right-hand side \( g(x,t) \) can often be expressed as a linear combination of terms such as
\[
\cos(\omega t),\ \cos(\omega x),\ \frac{\partial u}{\partial t},\ \frac{\partial^3 u}{\partial t \partial x^2},\ \delta(x-x_0),\ \mu(x-x_0) \sin u,\ p(x,t),
\]
where \( \mu(x) \) approximates the Dirac delta function or takes a small sinusoidal form, and \( p(x,t) \) may include discontinuities. When \( p(x,t) \) is modeled as a Wiener process, the existence of a compact random attractor has been shown \cite{Fan2004}. Perturbation theories for variations such as the double sine-Gordon equation have also been studied \cite{Popov2005}.

A hallmark feature of the sine-Gordon equation is its soliton solutions, localized waveforms that maintain their shape while propagating \cite{RyogoHirota1972}, \cite{Scott1973}. In the absence of external forces, the equation admits exact solutions including kinks, anti-kinks, and breathers. These have found applications in optical communication, condensed matter systems, and superconducting technologies. Various analytical techniques, such as the tanh method, rational exp-function method, and sine–cosine method, have been used to study these solutions \cite{Sickotra2021}, \cite{Rezazadeh2023}, \cite{IVANCEVIC2013}, \cite{WAZWAZ2006}.

To derive integrable forms or approximate behaviors, linearized models are frequently adopted \cite{BFL}, \cite{Bishop1986}, \cite{Kivshar1989}. Within the Hilbert space framework, the sine-Gordon equation has been rigorously analyzed in terms of existence, uniqueness, stability, optimal control, and attractor properties \cite{Ha-Na2}, \cite{GH2011}, \cite{Temam1988}, \cite{Shahruz2001}.

The rapid progress in quantum computing has further renewed interest in this equation \cite{Scott2000}. The function \( u(x,t) \), representing the phase difference, may be used in quantum logic operations such as phase gates, defined by \( R_{u(x,t)} = \text{diag}(1, e^{i u(x,t)}) \), or as a phase difference between qubit states, represented as \( u(x,t) = \theta_\beta(x,t) - \theta_\alpha(x,t) \) in the qubit state \( |\alpha| e^{i \theta_\alpha} |0\rangle + |\beta| e^{i \theta_\beta} |1\rangle \). This highlights the importance of understanding the behavior of \( u(x,t) \).

In practical modeling, especially in the context of Josephson junctions, solutions are often constructed as a known soliton profile \( \phi(x,t) \) plus a perturbation \( \epsilon \eta(x,t) \), written as  
\[
u(x,t) = \phi(x,t) + \epsilon \eta(x,t),
\]  
where \( \phi(x,t) \) satisfies the unperturbed equation and \( \eta(x,t) \) captures the influence of external perturbations.

This naturally raises the question of the conditions under which the soliton structure is preserved in the presence of such perturbations.

Since \( \eta \) reflects the influence of external forces, its boundedness is key to soliton stability. If \( \eta \) remains bounded over time, the soliton structure is likely preserved. Otherwise, unbounded growth may indicate collapse or a qualitative change in the solution. Thus, analyzing the boundedness of perturbations is essential to understanding soliton stability under realistic conditions.

This study develops a rigorous mathematical framework to analyze the boundedness and stability of the perturbation \( \eta \). However, as \( \eta \) is highly sensitive to the initial conditions of \( u(x,t) \), accurate determination of these conditions is crucial yet often challenging. To address this, we propose a data-driven approach using neural networks (NN) to estimate initial states from limited or noisy observations \cite{Brunton2019}.

The structure of this paper is as follows. Section 2 presents the theoretical groundwork for bounding perturbations using Lions’ variational method, depending on the regularity of the external force. The approach is rooted in the foundational techniques developed in \cite{GH2011}, \cite{Temam1988}, \cite{HaNa2002}.  
Section 3 examines numerical simulations to compare the effects of linearized versus nonlinear perturbation approaches \cite{Bishop1986}, \cite{Kivshar1989}.  
Section 4 introduces a neural network-based estimation technique for identifying the initial conditions of the sine-Gordon equation, demonstrating its effectiveness in improving perturbation analysis by accurately recovering the underlying initial states.

\section{Perturbation Analysis}

In this section, we analyze how perturbations influence the sine-Gordon soliton solutions and their stability. We consider both time-dependent and spatiotemporal perturbations and derive governing equations for the perturbation function. The perturbative approach allows us to assess the impact of external forces, boundary conditions, and small deviations in the initial conditions on soliton dynamics.

We begin our analysis by considering the perturbed form of the sine-Gordon equation in the presence of an external force
\begin{equation}\label{SGPer}  
    u_{tt} - u_{xx} + \sin u = \epsilon g(x,t).
\end{equation}

When \(g = 0\), equation \eqref{SGPer} admits an integrable solution, which we refer to as a soliton, denoted by $\phi$.
In fact, multiple soliton solutions exist in the sine-Gordon equation, including kink–kink, kink–antikink, static breather, and moving breather forms 
\cite{Sickotra2021}. These solitons can be derived using various analytical methods such as the tanh method, rational exp-function method, sech method, extended tan method, and sine–cosine method, providing a deeper mathematical understanding of their properties \cite{Rezazadeh2023}, \cite{WAZWAZ2006}.

A notable property of solitons is that they maintain their shape while traveling at a constant speed. For example, the soliton
\begin{equation}\label{soliton1}
\phi(x,t) = 4 \arctan \left[ \exp\left( \gamma (x -vt - x_0) \right) \right], \quad \gamma = (1-v^2)^{-1/2},
\end{equation}
satisfies the asymptotic conditions
\begin{equation}\label{soliton2}
\phi(x,t) \to 0 \quad \text{as} \quad x \to -\infty, \quad \text{and} \quad \phi(x,t) \to 2\pi \quad \text{as} \quad x \to \infty.
\end{equation}
The soliton in \eqref{soliton1} is a unique solution of the initial-boundary value problem
\begin{align*}
  &  \phi_{tt} - \phi_{xx} + \sin \phi =0, \quad -\infty < x < \infty, \quad  t >0, \\
  &  \phi(x,0) = 4 \tan^{-1} \left(e^{\gamma (x - x_0)} \right), \ \
    \phi_t(x,0) = \frac{-2 \gamma v }{\cosh(\gamma x)}, \quad -\infty < x < \infty, \\
  &  \phi(-\infty,t)=0, \quad \phi_x(\infty,t)= 2 \pi, \quad  t \ge 0.
\end{align*}
As an alternative boundary condition, we can impose
\[
\phi_x(-L,t) = \phi_x(L,t), \quad  t \ge 0,
\]
where $L$ is a positive number, which can be arbitrarily large or infinite.

For a small perturbation \(\epsilon g(x,t) \), we wish to express the solution of \eqref{SGPer} as
\begin{equation}\label{pertsol1-2}
    u(x,t) = \phi(x,t) + \epsilon \eta(x,t),
\end{equation}
where $\phi(x,t)$ is the unperturbed soliton, and $\epsilon \eta(x,t)$ represents a small perturbation with $0<\epsilon \le 1$. Substituting \eqref{pertsol1-2} into \eqref{SGPer}, we obtain
\begin{equation}\label{pertsol1-3}
     \eta_{tt} -  \eta_{xx} + \epsilon^{-1} (\sin(\phi+\epsilon \eta) - \sin \phi) = g(x,t).
\end{equation}
To solve for \( \eta(x,t) \) in (\ref{pertsol1-3}), we must specify appropriate initial and boundary conditions. These conditions depend on those of \( u \). It is reasonable to assume that \( \eta(x,t) \) inherits the conditions of \( u \), ensuring consistency with the original system's behavior.

In \cite{SALERNO1983}, a numerical solution of the perturbed sine-Gordon equation is provided for the system
\begin{align*}
 & u_{xx}-u_{tt}- \sin u = g(x,t, u), \\
 & u(x,0) = \phi(x,0), \ \ u_t(x,0) = \phi_t(x,0), \ \
  u_x(-13,t) = u_x(13,t) = 0, 
\end{align*}
where \( g(x,t,u) = \alpha u_t - \beta u_{xxt} + \gamma + \mu(x) \sin u \), with \( \mu(x) \) serving as an approximation of the Dirac delta function \( \delta(x) \). This formulation accounts for damping effects and localized forcing, offering insights into the behavior of the sine-Gordon equation under various external influences.

It is evident that \( u = \phi + \epsilon \eta \) is the solution of the system
\begin{align*}
 &   u_{tt} - u_{xx} + \sin u = \epsilon g(x,t), \\
  & u(x,0) = \phi(x,0), \ \ u_t(x,0) = \phi_t(x,0), \ \
   u_x(-L,t) = u_x(L,t) = 0, 
\end{align*}
if and only if \( \eta \) is the solution of the system 
\begin{align*}
& \eta_{tt} - \eta_{xx} + \epsilon^{-1} (\sin(\phi+\epsilon \eta) - \sin \phi) =  g(x,t), \\
& \eta(x,0) = \eta_t(x,0) = 0, \quad \eta_x(-L,t) = \eta_x(L,t) = 0.
\end{align*}
This reformulation, expressed as a deviation \( \eta(x,t) \), enables a direct analysis of how external forcing influences the soliton.

However, when oscillatory perturbations exist, such as \( g(x,t) = I_b + H = A \cos(\omega_1 t) + B \sin(\omega_2 x) \), the initial condition can incorporate both the soliton solution and oscillatory perturbations. For instance, the initial condition may take the form  
\[
u(x,0) = \phi(x,0) + \epsilon\cos(\omega_0 x).
\]  
In this case, the perturbation function \( \eta(x,t) \) satisfies  
\[
\eta(x,0) = \cos(\omega_0 x).
\]  
Therefore, if 
\[
u(x,0) = \phi(x,0) + \epsilon u_0(x), \ \ 
u_t(x,0) = \phi_t(x,0) + \epsilon v_0(x),
\]
then initial and boundary conditions for \( \eta \) are given by 
\begin{align*}
  &  \eta(x,0) = u_0(x), \quad \eta_t(x,0) = v_0(x), \quad -L < x < L, \\
  &  \eta(-L,t) = \eta(L,t) = 0 \quad \text{or} \quad \eta_x(-L,t) = \eta_x(L,t) = 0, \quad  t \geq 0.
\end{align*}  
These conditions ensure that the perturbation function \( \eta(x,t) \) appropriately captures the effects of initial oscillations and external forcing while preserving the necessary boundary constraints.

Since the equation is nonlinear, finding an explicit solution for \( \eta(x,t) \) is not straightforward. To simplify the analysis, one applies a first-order approximation for small \( \epsilon \), leading to the linearized equation 
\[
    \eta_{tt}  -  \eta_{xx} + \cos \phi \cdot \eta = g(x,t).
\]

In \cite{Fogel1977}, the perturbation \( \eta \) is assumed to be small, satisfying \( |\eta| < 1 \) with \( \epsilon = 1 \). They further assumed a harmonic time dependence of the form  
\(
    \eta(x,t) = f(x) e^{-i \omega t}
\)
and analyzed \( f(x) \) under the condition \( g(x,t) = 0 \).  Moreover, they investigated a case where the perturbation function includes a localized source, specifically  
\(
    g(x,t) = \delta(x + x_0) - \delta(x - x_0).
\)
This formulation allows for studying the impact of localized perturbations on soliton stability and dynamics. 

Now, we consider the initial and boundary problem  
\begin{align}
  &  \eta_{tt} -  \eta_{xx} + \epsilon^{-1}(\sin(\phi+\epsilon \eta) - \sin \phi)= g(x,t), \quad -L < x < L, \quad  t >0, \label{etaeq}\\
  &  \eta(x,0) = u_0(x), \ \ \eta_t(x,0) = v_0(x), \quad -L < x < L, \label{etaini}\\
  &  \eta(-L,t)=\eta(L,t)=0 \ \ \text{or} \ \ \eta_x(-L,t)=\eta_x(L,t)=0, \quad  t \ge 0. \label{etabd}
\end{align}

Taking \( \epsilon = 1 \) and \( \phi = 0 \) in (\ref{etaeq}), the equation reduces to the original sine-Gordon equation. Consequently, all theoretical results concerning the sine-Gordon equation remain valid under these conditions. It is well known that the existence, uniqueness, regularity, and stability of solutions depend on the initial and boundary conditions, as well as the external forcing with \( \epsilon = 1 \) and \( \phi = 0 \), see \cite{Ha-Na2}, \cite{Lionsop}.

Soliton solutions provide a fundamental framework for analyzing the boundedness and stability of perturbation terms \( \eta \). Studying the boundedness of \( \eta \) is crucial, as it determines whether the soliton structure is preserved under external influences. If the perturbation remains bounded, the soliton is likely to persist, whereas divergence may lead to its collapse. Thus, perturbation analysis plays a key role in quantitatively understanding the long-term behavior and stability of solitons, particularly in the presence of external forces and dissipative effects.

To formulate the equation as an ordinary differential equation, we introduce abstract Hilbert spaces. Let \( I = (-L, L) \) be an open interval in \( \mathbb{R} \). Define \( H = L^2(I) \) as the Hilbert space equipped with the inner product  
\(
(u, v) = \int_{-L}^{L} u(x) v(x) \,dx
\)  
and the associated norm  
\(
|u|=\sqrt{(u,u)}.
\)  
Here, \( L \) may be infinite.  

Let \( V \) be a Hilbert space that is continuously embedded in \( H \) and is equipped with an inner product \( ((u, v)) \) and norm \( \|u\| \). We consider the standard choices  
\[
V = H_0^1(I) \quad \text{or} \quad V = H_N^1(I) = \{ u \in H^1(I) \mid u_x(\pm L) = 0 \}.
\]  
The space \( V^* \) denotes the dual space of \( V \) with the dual pairing 
\(
\langle u, v \rangle, \quad u \in V^*, \ v \in V.
\)  
The choice of \( V \) determines the boundary conditions imposed on the system. When \( V = H_0^1(I) \), we assume homogeneous Dirichlet boundary conditions, while \( V = H_N^1(I) \) corresponds to Neumann boundary conditions. These function spaces provide the appropriate framework for analyzing weak solutions and their stability properties.

Let us define a bilinear form \( a(u, v) \) from \( V \times V \) to \( \mathbb{R} \) as  
\[
a(u, v) = (u_x, v_x).
\]  
Then, \( a(u, v) \) is symmetric and satisfies the boundedness condition  
\[
|a(u, v)| \leq \|u\|_V \|v\|_V, \quad \forall u, v \in V.
\]
Moreover, we can define the bounded linear operator \( \mathcal{A} \in \mathcal{L}(V, V^*) \) by the relation  
\[
a(u, v) = \langle \mathcal{A}u, v \rangle, \quad \forall u, v \in V.
\]  
This operator \( \mathcal{A} \) plays a fundamental role in the variational formulation of the problem and provides a rigorous framework for studying the existence, uniqueness, and stability of weak solutions in the Hilbert space setting.

The function \( f(t,\eta) \) is defined on \( H \) as  
\[
f(t,\eta)(x) = \epsilon^{-1} (\sin \phi(x,t) - \sin (\phi(x,t)+\epsilon \eta(x,t)) ) + g(x,t), \quad x \in I.
\]
Then, \( f(t, \eta) \) satisfies the Lipschitz continuity condition on \( H \)
\begin{equation}\label{fLips}
|f(t, \eta_1) - f(t, \eta_2)| \leq |\eta_1 - \eta_2|.  
\end{equation}
Since \( f(t,0)(x) = g(x,t) \), it follows that  
\begin{equation}\label{festimation}
|f(t, \eta)| \leq |\eta| + |g(t)|.    
\end{equation}

In such a Hilbert space setting, differentiation with respect to time is denoted by \( ' \). Using the notations introduced above, equations \eqref{etaeq}-\eqref{etabd} are reduced to a second-order equation in \( V^* \) of the form
\[
\eta'' + \mathcal{A}\eta = f(t, \eta), \ \
\eta(0)=u_0 \in V, \quad \eta'(0)=v_0 \in H. 
\]

The bilinear form \( a(u, v) \) is said to be coercive on \( V \times V \) if it satisfies the following positivity condition there exists a constant \( c > 0 \) such that  
\begin{equation}\label{acoercivity}
a(u, u) \geq c \| u \|^2, \quad \forall u \in V.
\end{equation}
When \( V = H_0^1(I) \), the norm \( \| u \| \) in \( V \) is equivalent to \( |u_x| \), which is induced by \( a(u, v) \). Hence, \( a(u, v) \) is coercive.  
However, when \( V = H_N^1(I) \), the norm \( \| u \| \) in \( V \) is not equivalent to \( |u_x| \), and therefore, \( a(u, v) \) is not coercive.  
As \( a(u,v) \) is not inherently coercive in general, coercivity can be ensured by adding a regularization term such as \( \epsilon u \). Therefore, without loss of generality, we can assume that \( a(u,v) \) is coercive. 

The boundedness and coercivity of \( a(u, v) \), along with the Lipschitz continuity of \( f(t, \eta) \), play a crucial role in establishing the well-posedness of the variational problem and ensuring the existence and uniqueness of solutions. 

Let us introduce the space of solutions as follows.  
\begin{equation*}\label{WOT}
W(0,T)=\{ u\ \ u\in L^2(0,T;V),\quad u' \in L^2(0,T;H),\quad u'' \in L^2(0,T;V^*)\},
\end{equation*}
where the derivatives are understood in the sense of distributions with the values in $H$ and $V^*$, see \cite{Lions1992}. Space $W(0,T)$ becomes a Hilbert space, when its inner product is set to be the sum of the inner products in the constituent spaces.

\begin{definition}\label{soldef}
Let $u_0\in V,\; v_0\in H,\; T>0$, and $f\in L^2(0,T;H)$. Function $u\in W(0,T)$ is called a \emph{weak solution} of the problem \eqref{etaeq}-\eqref{etabd}, if equation
 \begin{equation}\label{soldefeq1}
\eta'' + \mathcal{A}\eta = f(t, \eta)
\end{equation}
is satisfied in $V^*$ a.e. on $[0,T]$, and the initial conditions
 \begin{equation}\label{soldefeq2}
\eta(0)=u_0,\quad  \eta'(0)=v_0
\end{equation}
are satisfied in $V$ and $H$ correspondingly.
\end{definition}

Lemma \ref{vectortemam} is a crucial result concerning the regularity of solutions, originally stated as Lemma 2.4.1 in \cite{Temam1988}. 
It establishes continuity properties of weak solutions and provides an essential energy identity.

\begin{lemma}\label{vectortemam} 
Suppose that $u\in W(0,T)$, and
$u''+\mathcal{A}u\in L^2(0,T;H)$. Then, after a modification on a set
of measure zero, $u\in C([0,T];V)$, $u'\in C([0,T];H)$ and, in the
sense of distributions on $(0,T)$ one has
\begin{equation}
(u''+ \mathcal{A}u,u')=\frac 12 \frac {d}{dt}\{|u'|^2+\|u\|^2\}.
\end{equation}
\end{lemma}

\begin{lemma}\label{lemma1}
Let $g\in L^2(0,T;H)$ and let $\eta$ be a weak solution of the problem \eqref{soldefeq1}-\eqref{soldefeq2}. Then $\eta\in C([0,T];V),\; \eta'\in C([0,T];H)$, and
\begin{equation}\label{lemma1eq1}
|\eta'(t)|^2+\|\eta(t)\|^2\leq c\left(|v_0|^2+\|u_0\|^2+\int_0^t |g(s)|^2\, ds\right)
\end{equation}
for any $t\in[0,T]$. The constant $c$ is dependent only on $T$.
\end{lemma}
\begin{proof}
Since \( f \in L^2(0,T;H) \), it follows that \( u'' + \mathcal{A}u \in L^2(0,T;H) \). By applying Lemma \ref{vectortemam}, we obtain the following energy identity:  
\begin{equation}\label{lemma1pf1}
\frac{1}{2} \frac{d}{dt} \left( |\eta'|^2 + a(\eta, \eta) \right) = (f(t,\eta), \eta).
\end{equation}
By using the estimate in (\ref{festimation}), we bound the right-hand side as  
\begin{equation}\label{lemma1pf2}
|(f(t,\eta), \eta)| \leq |\eta| |\eta'| + |g(t)| |\eta'|.
\end{equation}
Integrating (\ref{lemma1pf1}) over \( [0,t] \), we obtain  
\begin{equation}\label{lemma11pf3}
|\eta'(t)|^2 + a(\eta(t), \eta(t)) = |v_0|^2 + a(u_0, u_0) + \int_0^t (f(s,\eta(s)), \eta(s)) ds.
\end{equation}
This provides an essential estimate for the stability and boundedness of \( \eta \), which depends on the initial conditions and the external forcing term \( g(x,t) \). From \eqref{lemma11pf3}, we can conclude that \( \eta \in C([0,T];V) \) and \( \eta' \in C([0,T];H) \), ensuring the continuity of the solution in the respective function spaces.

Applying (\ref{acoercivity}) and (\ref{lemma1pf2}) to (\ref{lemma11pf3}), we obtain  
\begin{align*}
 & |\eta'(t)|^2 + c \|\eta(t)\|^2
  \leq |v_0|^2+\|u_0\|^2 + 
\int_0^t |\eta(s)| |\eta'(s)| + |g(s)| |\eta'(s)| \, ds \\
& \leq |v_0|^2+\|u_0\|^2
 + \int_0^t \|\eta(s)\|^2 +|\eta'(s)|^2 ds +\int_0^t |g(s)|^2 ds. 
\end{align*}
Here, we used the fact that \( V \) is continuously embedded in \( H \), ensuring the validity of the norm estimates.  
Applying Gronwall’s inequality, we conclude (\ref{lemma1eq1}), which establishes the boundedness of the solution \( \eta \).
\end{proof}

Lemma \ref{lemma1} implies that the solution 
$\eta$ is constrained by the initial conditions and the external force. The smaller the influence of these factors, the smaller the influence on the solution.

\begin{lemma}\label{lemma2}
Let \( g_i \in L^2(0,T;H) \), and let \( \eta_i \) (\( i=1,2 \)) be the solutions of the problem:
\begin{align*}
& \eta_i'' + \mathcal{A} \eta_i = f_i(t, \eta), \quad
  \eta_i(0) = u_{0,i} \in V, \quad \eta_i'(0) = v_{0,i} \in H.
\end{align*}
Then, for any \( t \in [0,T] \), the following inequality holds:
\begin{multline}\label{lemma2eq1}
|\eta_2'(t) - \eta_1'(t)|^2 + \| \eta_2(t) - \eta_1(t) \|^2 \\
\leq c \big( |v_{0,2} - v_{0,1}|^2 + \|u_{0,2} - u_{0,1}\|^2 
+ \|g_2 - g_1\|^2_{L^2(0,T;H)} \big).
\end{multline}

Furthermore, the solution of the problem \eqref{soldefeq1}--\eqref{soldefeq2} is unique.
\end{lemma}

\begin{proof}
Define \( z = \eta_2 - \eta_1 \), then \( z \) satisfies the equation
\[
z'' + \mathcal{A}z = f_2(t, \eta_2) - f_1(t, \eta_1).
\]
The difference \( f_2(t, \eta_2) - f_1(t, \eta_1) \) can be estimated as follows.
\begin{align*}
 & |f_2(t,\eta_2)-f_1(t,\eta_1)| \le \epsilon^{-1} |\sin (\phi(t)+\epsilon \eta_2(t)) - \sin (\phi(t)+\epsilon \eta_1(t))| + |g_2(t)-g_1(t)| \\
& \leq |\eta_2(t)-\eta_1(t)| + |g_2(t)-g_1(t)| = |z| + |g_2(t)-g_1(t)|.
\end{align*}

Multiplying the equation by \( z' \) and estimating each inner product as in Lemma \ref{lemma1}, we obtain \eqref{lemma2eq1}. 

The uniqueness of the solution follows directly from \eqref{lemma2eq1}, as it implies that any two solutions with the same initial conditions and external force must be identical.
\end{proof}

\begin{theorem}\label{Thm1}
Assume that \( g \in L^2(0,T;H) \). Then, there exists a unique weak solution \( \eta \in W(0,T) \) to the problem \eqref{soldefeq1}-\eqref{soldefeq2}. Furthermore, the solution satisfies the following regularity properties
\[
\eta \in C([0,T];V), \quad \eta' \in C([0,T];H), \quad \eta'' \in L^2(0,T;V^*).
\]
\end{theorem}  

\begin{proof}
The theorem can be proved by modifying the proof in \cite{Ha-Na2}. The proof proceeds by first constructing an approximate solution as a finite linear combination of basis functions in \( V \), which also form an orthonormal basis in \( H \) with respect to the operator \( \mathcal{A} \). Then, using Lemma \ref{lemma1}, we establish the uniform boundedness of these approximate solutions. Next, by utilizing the weak compactness property of the space, we extract a weakly convergent subsequence. Finally, we take the limit of this sequence to obtain a weak solution, ensuring the convergence of the nonlinear term by leveraging the compact embedding of \( V \) into \( H \). This approach guarantees the existence of a weak solution satisfying the required regularity conditions.
\end{proof}

In the case where \( g \in L^2(0,T;V^*) \), we consider \( W_r(0,T) \), the Hilbert space of more regular weak solutions, defined as  
\[
W_r(0,T) = \left\{ u \mid u \in L^2(0,T;V), \, u' \in L^2(0,T;V), \, u'' \in L^2(0,T;V^*) \right\}.
\]

It is well known that \( g = \delta(x - x_0) \in L^2(0,T;V^*) \). Let us consider a nonlinear function as in \cite{SALERNO1983}  
\[
g(x,t, \eta(x,t)) = \delta(x - x_0) \sin \eta(x,t).
\]
Assuming that the solution \( \eta \in W_r(0,T) \subset L^2(0,T;V) \), the term involving the delta function,  
\[
\int_{-L}^L \delta(x - x_0) \sin \eta(x,t) v(x) \,dx = \sin \eta(x_0,t) v(x_0),
\]
is well-defined since \( \eta(t), v \in V \subset C([-L,L]) \), where the continuous embedding ensures pointwise continuity.

Accordingly, for \( v, w \in V \), define the functional \( \title{g} \) by  
\[
\langle \title{g}(w), v \rangle = \sin w(x_0) v(x_0), \quad x_0 \in [-L,L].
\]
Using the norm equivalence from the embedding \( V \subset C([-L,L]) \), we obtain 
\[
|\langle \title{g}(w), v \rangle| \leq \max_{x \in [-L,L]} |w(x)| \max_{x \in [-L,L]} |v(x)| \leq c \|w\| \|v\|,
\]
which confirms that \( \title{g} : V \to V^* \) is a well-defined operator. Hence, \( g(\eta) \in  L^2(0,T;V^*) \). We can easily prove that \( g(\eta) \) is Lipschitz continuous. For \( \eta_1, \eta_2 \in V \), we have  
\[
\langle g(\eta_1) - g(\eta_2), v \rangle = (\sin \eta_1(x_0) - \sin \eta_2(x_0)) v(x_0).
\]
Using the Lipschitz property of the sine function,  
\[
|\sin \eta_1(x_0) - \sin \eta_2(x_0)| \leq |\eta_1(x_0) - \eta_2(x_0)|,
\]
it follows that  
\[
|\langle g(\eta_1) - g(\eta_2), v \rangle| \leq |\eta_1(x_0) - \eta_2(x_0)| |v(x_0)|.
\]
Applying the embedding \( V \subset C([-L,L]) \),  
\[
|\langle g(\eta_1) - g(\eta_2), v \rangle| \leq c \|\eta_1 - \eta_2\| \|v\|,
\]
which confirms that \( g(\eta) \) is Lipschitz continuous in \( V^* \).

In this case, where \( g(\eta) \in  L^2(0,T;V^*) \), we can guarantee the well-posedness of solutions by introducing an additional perturbation term \( u_{xxt} \), as discussed in \cite{GH2011}.  

Lemma \ref{settemammodified} plays a similar role to Lemma \ref{vectortemam}, providing essential regularity and energy estimates for the solution. Specifically, it ensures the boundedness of the perturbation function \( \eta \) and establishes continuity properties in appropriate function spaces. The proof of this result is detailed in \cite{GH2011} for the case where \( \epsilon = 1 \) and \( \phi = 0 \).

\begin{lemma}\label{settemammodified}
Let $u\in W_r(0,T)$. Then, after a modification on the set of measure zero, 
$u\in C([0,T];V)$, $u'\in C([0,T];H)$ and, in the sense of distributions on $(0,T)$ one has
\begin{equation}\label{setwwp1}
\frac {d}{dt}||u||^2=2((u',u))=2\langle  \mathcal{A}u,u'\rangle,\quad \text{and}\quad \frac {d}{dt}|u'|^2=2\langle u'',u'\rangle.
\end{equation}
\end{lemma}

\begin{theorem}\label{Thm2}
Let \( \lambda \) be a positive real number, and assume that \( f(t,\eta) \) is Lipschitz continuous from \( L^2(0,T;V) \) into \( L^2(0,T;V^*) \). Then, there exists a unique weak solution \( \eta \) satisfying  
\begin{align}
  &       \eta'' + \lambda \mathcal{A}\eta' + \mathcal{A}\eta = f(t, \eta), \label{thm2eq1} \\
  &  \eta(0) = u_0 \in V, \quad \eta'(0) = v_0 \in H. 
  \label{thm2eq2}
\end{align}
The solution satisfies the regularity properties
\[
\eta \in C([0,T];V), \quad \eta' \in C([0,T];H), \quad \eta'' \in L^2(0,T;V^*).
\]  
Moreover, for all \( t \in [0,T] \), the following energy estimate holds
\begin{equation} \label{thm2energy}
|\eta'(t)|^2 + \| \eta(t) \|^2  + \int_0^t \| \eta'(s) \|^2 \,ds
\leq C \left( \| u_0 \|^2 + |v_0|^2  + \|g\|_{L^2(0,T;V^*)}^2 \right),
\end{equation}  
where \( C \) is a constant depending on \( T \) and \( \lambda \).
\end{theorem} 

\begin{proof}
The proof can be carried out in a manner similar to the arguments presented in the preceding lemmas and theorems.
\end{proof}

Now, we show that choosing more regular initial conditions in (\ref{thm2eq2}) implies more regular weak solutions.
The operator \( \mathcal{A} : V \to V^* \) was defined by  
\[
\langle \mathcal{A}u, v \rangle = \int_I  u_x  v_x \, dx, \quad u,v\in V.
\]
Define the domain of \( \mathcal{A}\) as  
\(
D(\mathcal{A}) = \{ v \in V \mid | \mathcal{A}v| < \infty \}.
\)  
The norm of \( v \in D(\mathcal{A}) \) is given by \( |\mathcal{A}v| \), making \( D(\mathcal{A}) \) a Hilbert space. Moreover, \( \mathcal{A} : D(\mathcal{A}) \to H \) is an isometry, as established in \cite{Temam1988}.  
Since \( D(\mathcal{A}) \) is densely and compactly embedded in \( V \), we can consider the Gelfand triple  
\(
D(\mathcal{A}) \subset V \subset H,
\)
where \( V \) is identified with its dual \( V^* \), and \( [D(\mathcal{A})]' \) is identified with \( H \). Within this framework, the arguments of this section lead to the following theorem.  

\begin{theorem} \label{thm3} 
Let \( \eta \) be the weak solution of the equation  
\begin{equation*}
\eta'' + \mathcal{A}\eta = f(t, \eta), 
\ \ \eta(0) = u_0 \in V, \quad \eta'(0) = v_0 \in H.  
\end{equation*}  
If the initial data and and the external force satisfy  
\begin{equation*}
u_0 \in D(\mathcal{A}), \quad v_0 \in V, \quad f(t, \eta) \in L^2(0,T;V),  
\end{equation*}  
then the solution satisfies  
\begin{equation*}
\eta \in C([0,T];D(\mathcal{A})), \quad \eta' \in C([0,T];V), \quad \eta'' \in L^2(0,T;H).  
\end{equation*}  
\end{theorem}  

This theorem states that if the initial conditions and external forcing have additional smoothness, then the solution exhibits higher regularity, ensuring continuity and differentiability properties in stronger function spaces. 

\section{Numerical solutions}

In this section, we set \( x_0 = 0 \) and consider the soliton solution given by
\[
\phi(x,t) = 4 \tan^{-1} \left( e^{\gamma (x - v t)} \right), \quad \gamma = \frac{1}{\sqrt{1-v^2}}.
\]

Let \( u \) be the solution to the perturbed sine-Gordon equation
\begin{align*}
 & u_{tt} - u_{xx} + \sin u = \epsilon\left(A\cos(n\pi x/L)+ B\cos(n\pi t/T)\right), \\
 & u_x(-L,t) = u_x(L,t) = 0, \\
 & u(x,0) = \phi(x,0), \quad u_t(x,0) = \phi_t(x,0).
\end{align*}

Let \( \eta \) be the perturbed solution of 
\begin{align*}
  & \eta_{tt} - \eta_{xx} + \epsilon^{-1}(\sin(\phi+\epsilon \eta) - \sin \phi) = A\cos(n\pi x/L)+ B\cos(n\pi t/T), \\
  & \eta(x,0) = \eta_t(x,0) = 0, \quad
  \eta_x(-L,t) = \eta_x(L,t) = 0. 
\end{align*}

We fix the spatial domain size as \( L = 13 \), the time horizon as \( T = 20 \), and set the forcing parameters as \( A = 1 \), \( B = 2 \), and \( n = 4 \).

Since \( \phi(x,t) \) is a soliton solution satisfying the boundary conditions \( \phi(-\infty,t) = 0 \) and \( \phi(\infty,t) = 0 \), the use of finite boundary conditions \( u_x(-L,t) = u_x(L,t) = 0 \) introduces negligible error. The soliton’s rapid spatial decay ensures these boundaries effectively approximate the behavior at infinity.

Fig. \ref{fig1} presents the solution \( u(x,t) \) of the original sine-Gordon equation and the perturbed variant, evaluated at \( t = 2, 5, 8, 10, 15 \), along with a corresponding colormap plotted over the time interval \([0,20]\).  

Fig. \ref{fig2} shows the solution of the perturbed nonlinear equation,  
\[
u(x,t) = \phi(x,t) + \epsilon \eta(x,t),
\]
with \( \epsilon = 0.05 \), at the same time points.

The two results closely match the solutions \( u(x,t) \) of both the original and perturbed sine-Gordon equations, demonstrating that the perturbation model accurately captures the system’s behavior over time. 

\begin{figure}[H]
    \centering
    \begin{minipage}{0.49\textwidth}
        \centering
        \includegraphics[width=\textwidth]{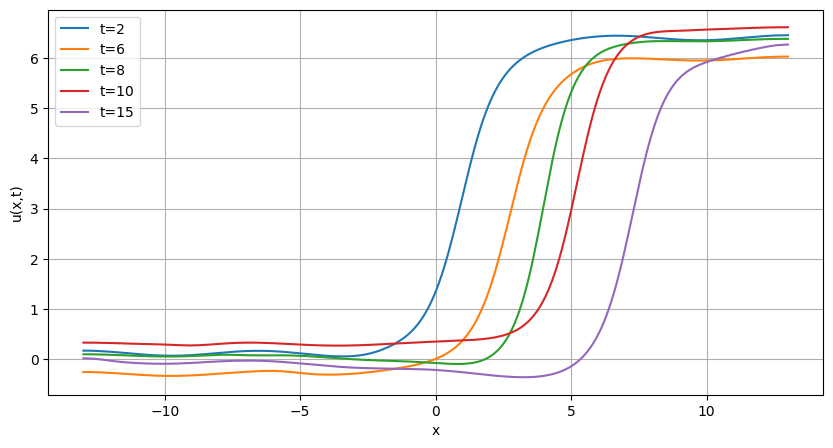}
    \end{minipage}
    \hfill
    \begin{minipage}{0.49\textwidth}
        \centering
        \includegraphics[width=\textwidth]{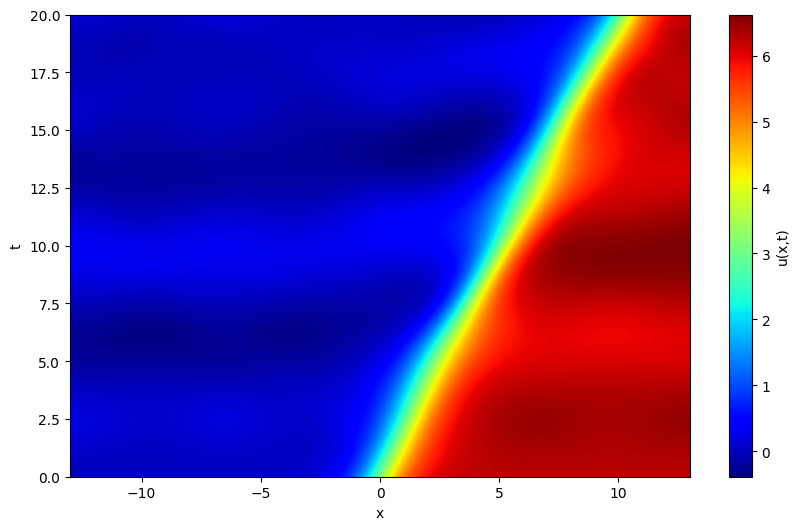}
    \end{minipage}
    \caption{Graphs of the original sine-G equation}
    \label{fig1}
\end{figure}
\begin{figure}[H]
    \centering
    \begin{minipage}{0.49\textwidth}
        \centering
        \includegraphics[width=\textwidth]{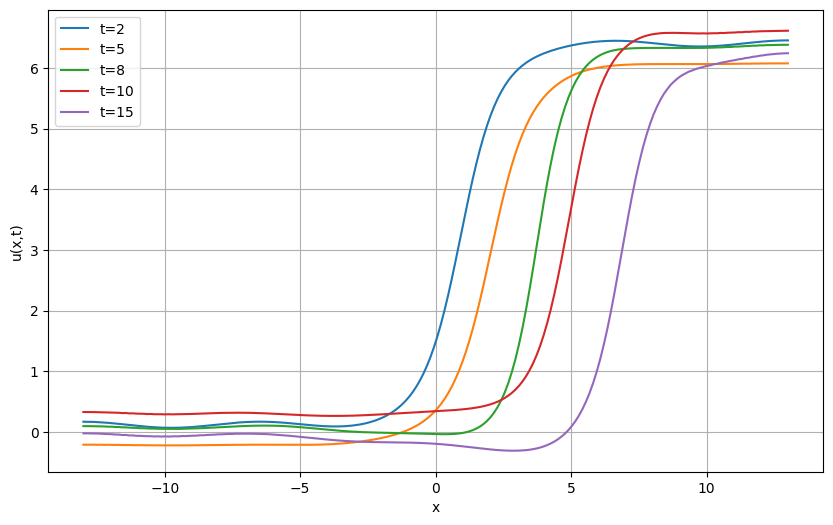}
    \end{minipage}
    \hfill
    \begin{minipage}{0.49\textwidth}
        \centering
        \includegraphics[width=\textwidth]{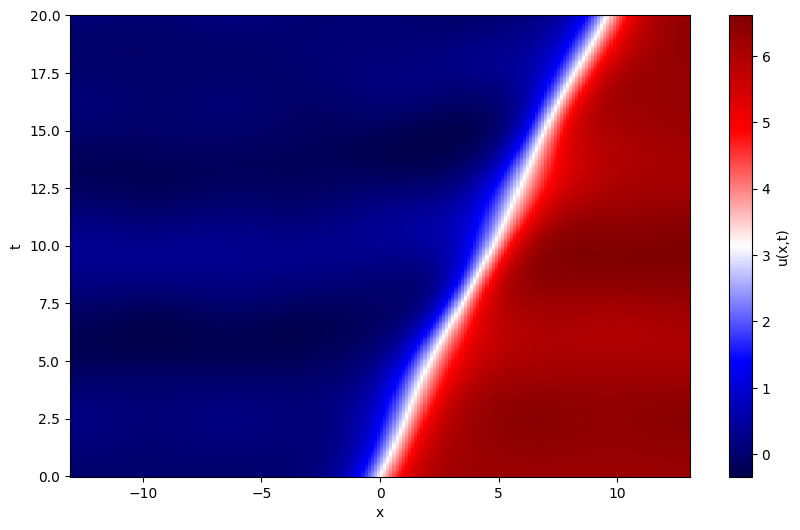}
    \end{minipage}
    \caption{Graphs of the perturbed sine-G equation}
    \label{fig2}
\end{figure}

In Fig. \ref{fig3} shows the solutions of the linearized equation for two cases: the left colormap corresponds to \( n = 1 \), and the right to 
\( n = 4 \), based on the following formulation
\begin{align*}
 & \eta_{tt} - \eta_{xx} + \cos \phi \cdot \eta = \cos(n\pi x/T)+ 2\cos(n\pi t/T), \\
 & \eta_x(-L,t) = \eta_x(L,t) = 0, \quad \eta(x,0) = \eta_t(x,0) = 0.
\end{align*}
The linearized equation exhibits high sensitivity to the frequency of the perturbation term. When \( n = 4 \), the resulting colormap closely resembles those of Figs. \ref{fig1} and \ref{fig2}, indicating good agreement with the nonlinear system. However, for \( n = 1 \), the response deviates significantly, showing agreement only in localized regions. To improve the fidelity of the approximation, particularly over longer time intervals, a higher-order Taylor expansion in \( \epsilon \) is required for a more accurate representation of the system’s behavior.

\begin{figure}[H]
    \centering
    \begin{minipage}{0.49\textwidth}
        \centering
        \includegraphics[width=\textwidth]{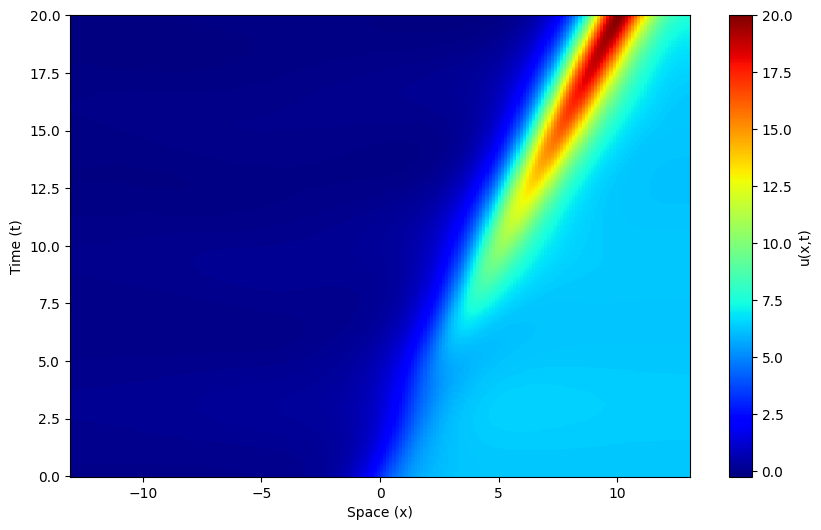}
    \end{minipage}
    \hfill
    \begin{minipage}{0.49\textwidth}
        \centering
        \includegraphics[width=\textwidth]{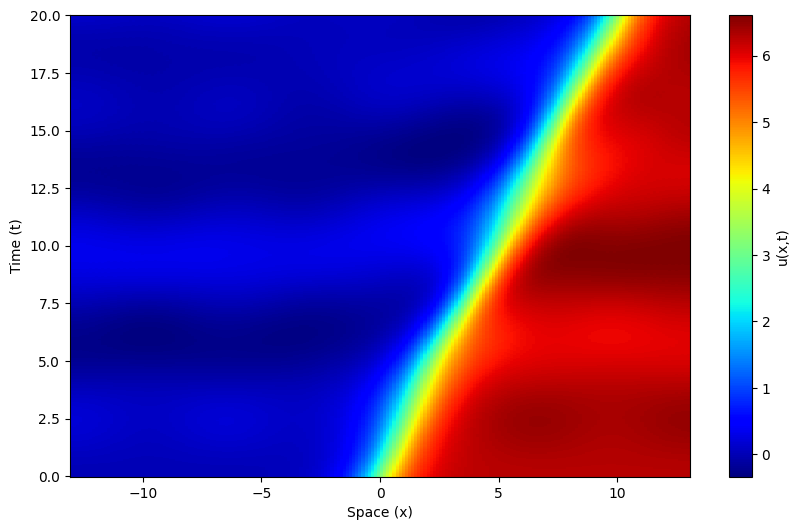}
    \end{minipage}
    \caption{Graphs of the linearized sine-G equation}
    \label{fig3}
\end{figure}

Fig. \ref{fig4} presents the results when the initial conditions are modified to \( \eta(x, 0) = \cos(4\pi x / L) \) and \( \eta_t(x, 0) = -4\pi / L \sin(4\pi x / L) \). The sinusoidal form of the initial data induces pronounced oscillatory behavior, particularly near the region where $u(x,t) \approx 0$.

\begin{figure}[H]
    \centering
    \begin{minipage}{0.49\textwidth}
        \centering
        \includegraphics[width=\textwidth]{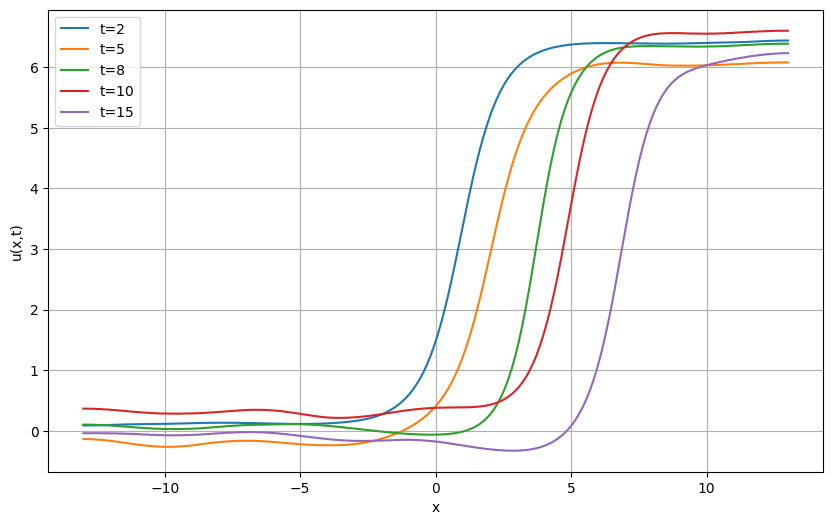}
    \end{minipage}
    \hfill
    \begin{minipage}{0.49\textwidth}
        \centering
        \includegraphics[width=\textwidth]{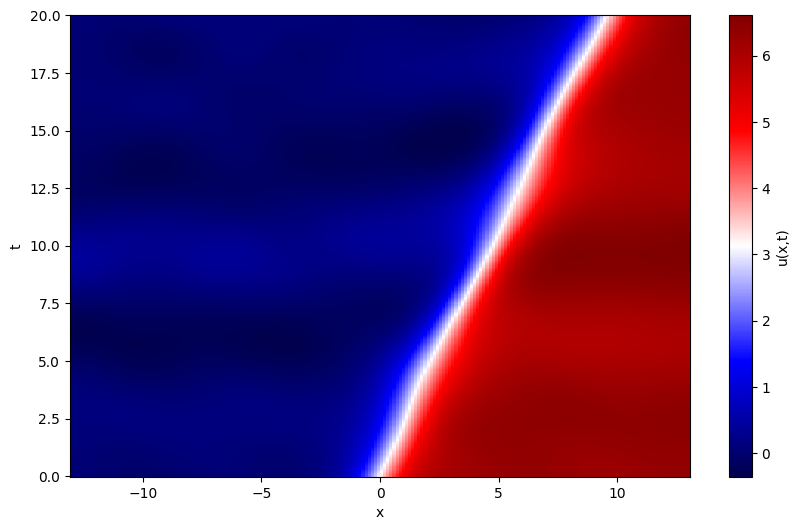}
    \end{minipage}
    \caption{Graphs of the perturbed sine-G equation}
    \label{fig4}
\end{figure}


\section{Neural Network-Based Estimation of Initial Conditions}

We address the inverse problem of determining the initial conditions \( u_0(x) \) and \( v_0(x) \) for the nonlinear damped sine-Gordon equation
\[
\eta_{tt} - \eta_{xx} + \epsilon^{-1}(\sin(\phi + \epsilon \eta) - \sin \phi) = g(x,t),
\]
posed on the spatial domain \( [-L, L] \) with \( L = 13 \), and over the time interval \( [0, T] \) with \( T = 20 \). The system is equipped with homogeneous Neumann boundary conditions, and the initial conditions are specified as
\[
\eta(x,0) = u_0(x), \quad \eta_t(x,0) = v_0(x).
\]
Assuming the forcing term \( g \in L^2(0,T;V) \), and the initial data satisfy \( u_0 \in D(\mathcal{A}) \), \( v_0 \in V \), Theorem \ref{thm3} ensures that the solution \( \eta \) has the following regularity
\[
\eta \in C([0,T];D(\mathcal{A})), \quad \eta' \in C([0,T];V), \quad \eta'' \in L^2(0,T;H).
\]

We define the admissible set of initial data as  
\[
Q_{ad} = D(\mathcal{A}) \times V,
\]  
which is compactly embedded in \( V \times H \), and hence also in \( L^2(0,T;H) \times L^2(0,T;H) \) due to the solution’s regularity. This compact embedding guarantees the continuity of the solution map \( q \mapsto \eta(q) \), and the weak lower semicontinuity of the objective functional
\[
J(q) = \|\eta(q) - z_1\|_{L^2(0,T;H)}^2 +  \|\eta_t(q) - z_2\|_{L^2(0,T;H)}^2,
\]
where \( q = (u_0, v_0) \in Q_{ad} \), and \( z_1 \), \( z_2 \) denote the desired solution and its time derivative, respectively, observed via the relation \( \epsilon^{-1}(u - \phi) \).  

Thus, the problem reduces to minimizing \( J(q) \) over \( Q_{ad} \), and the compactness of \(Q_{ad} \) guarantees the existence of a minimizer $\hat{q} \in Q_{ad}$. To guarantee the uniqueness of the minimizer \( \hat{q} \), it is sufficient to add a regularization term involving \( q \) to the objective functional \( J(q) \). The resulting functional becomes strictly convex, ensuring the existence of a unique optimal initial condition.

To solve this inverse problem numerically, we adopt a data-driven approach. Instead of directly solving the PDE-constrained optimization problem, we train a neural network using solution data generated from synthetic initial conditions. Specifically, we generate 50 initial profiles parameterized by \( \omega \in [n \pi /L-0.5, n \pi /L+0.5] \)
\[
u_0(x;\omega) = \cos(\omega x), \quad v_0(x;\omega) = -\omega \sin(\omega x).
\]
For each \( \omega \), the forward PDE is solved using the central difference method. The resulting solutions \( \eta(x,t;\omega) \) and \( \eta_t(x,t;\omega) \) are sampled at selected space-time points to construct the training dataset consisting of input-output pairs
\[
(x, \omega, t) \mapsto (\eta(x,t), \eta_t(x,t)),
\]
which the neural network is trained to reproduce.

In the training dataset, only the data corresponding to the target frequency 
$\omega=n \pi /L$ are perturbed with Gaussian noise. This models realistic observational uncertainty while ensuring that the network learns from clean PDE-generated trajectories at other frequencies.

The network \( \mathcal{F}_\theta : \mathbb{R}^3 \to \mathbb{R}^2 \) is a fully connected feedforward model with trainable parameters \( \theta \). It consists of an input layer of size 3, followed by four hidden layers of widths 64, 128, 64, and 64, with ReLU activation. The output layer yields two values \( (\hat{\eta}, \hat{\eta}_t) \). The network is trained using the Adam optimizer with a learning rate of \( 10^{-3} \), for up to 100,000 epochs with full-batch learning. The loss function is defined as
\[
\mathcal{L}(\theta) = \text{MSE}[\hat{\eta}, \eta_{\text{data}}] +  \text{MSE}[\hat{\eta}_t, \eta_{t,\text{data}}],
\]
balancing accuracy across both components.

In our setting, the external force is chosen as \( g(x,t) =\cos(n \pi x/L)+ 2\cos(n \pi t/T) \), which belongs to \( L^2(0,T;V) \). Under this condition, \( \eta \) and \( \eta_t \) are continuous in \( (x,t) \), justifying the pointwise loss function \( \mathcal{L}(\theta) \). 

Training includes early stopping based on a loss threshold or plateauing. The implementation is carried out in PyTorch with GPU acceleration.

A key feature of this approach is that the PDE is used only for data generation, not embedded into the network itself. This makes it a fully data-driven method, in contrast to physics-informed approaches, and allows application in cases where governing equations are partially unknown or costly to evaluate.

After training, the network is evaluated at \( t = 0 \) and fixed \( \omega = n \pi /L \) to reconstruct the initial conditions
\[
u_{0NN}(x) = \hat{\eta}(x, \omega=n \pi /L, t=0; \theta), \quad v_{0NN}(x) = \hat{\eta}_t(x, \omega=n \pi /L, t=0; \theta).
\]
These estimated functions are not explicitly trained but arise implicitly from the optimized parameters \( \theta \). Consequently, the inverse problem is reformulated as the following unconstrained optimization
\[
\min_{\theta} \, \mathcal{L}(\theta),
\]
where the learned initial condition is recovered by evaluating the trained network. The compact embedding \( Q_{ad} \hookrightarrow V \times H \hookrightarrow L^2(0,T;H) \) ensures that the output remains within the proper admissible function space.

The performance of the network depends on the data volume more data improve accuracy but increase cost, while fewer data reduce accuracy. Since the goal is to estimate spatial functions \( u_0(x), v_0(x) \), spatial resolution is more critical than temporal. In practice, sampling over the full spatial domain at only two time points, \( t = 0 \) and \( t = \Delta t \), is sufficient. Additional time samples can improve performance but also raise computational costs.

In our simulations, the spatial domain is discretized with \( N_x = 201 \) points and the time domain with \( N_t = 386 \) steps, using a CFL number of 0.2. This corresponds to \( \Delta t = 0.5020 \) and \( \Delta x = 0.1300 \). Smaller values of 
$n$  lead to faster convergence in identifying the initial conditions, whereas larger values significantly increase the computational load. For this reason, we set $n=4$ in this study.

Table \ref{tabloss_summary} presents the downsampling indices along with the corresponding values of individual loss terms, the number of epochs, and the mean squared errors (MSE) for the estimated initial conditions. Here, \( N(t) \) denotes the number of selected time points, while \( N(x) \) represents the number of spatial sampling points used during training. The training process was halted when the total loss dropped below the threshold of $10^{-3}$. For more accurate estimations, this threshold can be further lowered, though at the expense of increased computational cost.

\begin{table}[h!]
\centering
\caption{Downsampling data and corresponding loss values}
\begin{tabular}{|c|c|c|c|c|c|c|c|}
\hline
N(t) & N(x) & \text{Epoch} &  \( \text{Loss}_{\eta} \) & \( \text{Loss}_{\eta_t} \) & \( \text{MSE}_{u_0} \) & \( \text{MSE}_{v_0} \) & Noise \\
\hline
2  & 50 & 24002 &  0.000368 & 0.000829 & 0.000457 & 0.000557 & 0 \\
3  & 50 & 27987 &  0.000559 & 0.000910 & 0.000418 & 0.000590 & 0 \\
5  & 50 & 38981 &  0.000495 & 0.001086 & 0.000421 & 0.000642 & 0 \\
3  & 50 & 42436 &  0.000819 & 0.001331 & 0.000571 & 0.000791 & 0.05 \\
5  & 50 & 27303 &  0.000756 & 0.001061 & 0.000484 & 0.000784 & 0.05 \\
\hline
\end{tabular}
\label{tabloss_summary}
\end{table}

As shown in Fig. \ref{fig5}, the estimated initial conditions \( u_0(x) \) and \( v_0(x) \) are obtained from data at \( N(t) = 3 \) time points with a noise level of 0.05. Interestingly, the training converges faster in this noisy case than in the noise-free case, possibly because of a regularization effect of the noise.

\begin{figure}[H]
\centering
\includegraphics[width=\textwidth]{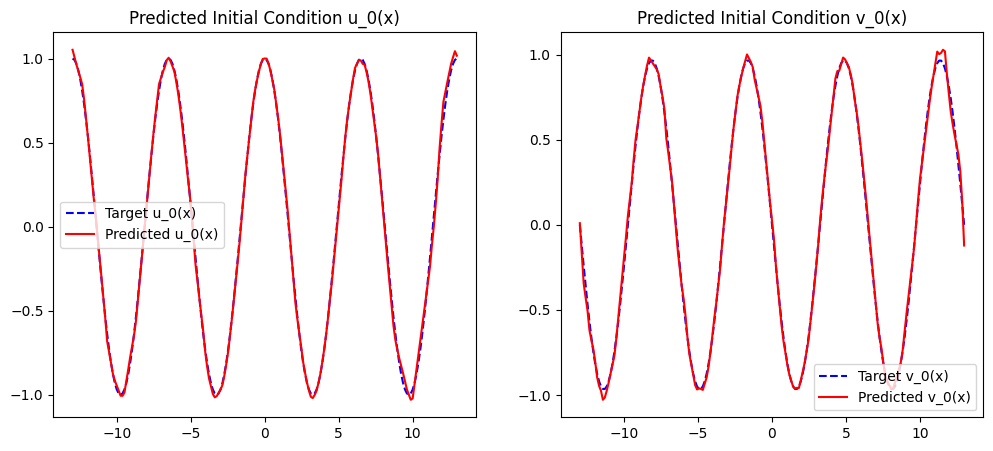}
\caption{The estimated initial conditions \( u_0(x) \) and \( v_0(x) \)obtained from data sampled at \( N(t) = 3 \) time points, under a noise level of 0.05.}
\label{fig5}
\end{figure}

\section{Conclusion}

This study examined the effects of perturbations on the stability of solitons within the framework of the sine-Gordon equation. A rigorous mathematical framework was established to analyze the boundedness and stability of perturbative solutions under external influences such as damping and driving forces. Numerical simulations confirmed that the solutions behave well under small perturbations.

In addition, a neural network-based approach was proposed to estimate the initial conditions, thus improving the robustness and accuracy of the perturbation analysis. Although the unknown initial functions in this study were limited to two sinusoidal components in space and time, a broader class of initial conditions can be represented as linear combinations of the operator's eigenfunctions \( \mathcal{A} \), which offers greater flexibility to model diverse initial profiles.

In general, the findings provide deeper insight into the dynamics of solitons and underscore their potential applications in quantum computing and related fields.

\end{document}